\documentclass[reqno]{amsart}

\usepackage{amsmath}
\usepackage{mathtools}
\usepackage{amssymb}
\usepackage{amsthm}
\usepackage{bbm}
\usepackage{dsfont}
\usepackage{mathrsfs}
\usepackage{commath}
\usepackage{breqn}

\usepackage{graphicx}
\usepackage{pdfpages}

\newtheorem{theorem}{Theorem}[section]
\newtheorem{lemma}{Lemma}[section]

\newtheorem{corollary}{Corollary}[section]

\theoremstyle{definition}

\theoremstyle{remark}
\newtheorem{remark}{Remark}[section]

\numberwithin{equation}{section}

\begin{document}

\title[Distribution of Permutation Matrix Entries Under Randomized Basis]{The Distribution of Permutation Matrix Entries Under Randomized Basis}

\author{Benjamin Tsou}
\address{Department of Mathematics, University of California, Berkeley, CA 94720-3840}
\email{benjamintsou@gmail.com}

\begin{abstract}

We study the distribution of entries of a random permutation matrix under a ``randomized basis,'' i.e., we conjugate the random permutation matrix by an independent random orthogonal matrix drawn from Haar measure.  It is shown that under certain conditions, the linear combination of entries of a random permutation matrix under a ``randomized basis'' converges to a sum of independent variables $sY+Z$ where $Y$ is Poisson distributed, $Z$ is normally distributed, and $s$ is a constant.

\vspace{12pt}

\vspace{12pt}

\end{abstract}

\maketitle
\allowdisplaybreaks

\section{Introduction}
Traditionally, random matrix theory has focused primarily on the eigenvalue distributions of matrices drawn from various measures (see e.g. \cite{diaconisgibbs} for a quick survey or \cite{andersonguionnet} for an introduction to the field).  However, there has also been interest in studying statistics related to entries of random matrices and specifically Haar-distributed orthogonal matrices.  Historically, one of the earliest results in this direction is due to Borel \cite{borel} who showed that the first entry of an $n \times n$ Haar-distributed orthogonal matrix scaled by $\sqrt{n}$ converges in distribution to a standard normal.  Diaconis and Freedman \cite{diaconisfreedman} strengthened this to obtain total variation convergence of the first $k$ coordinates of a random vector on the sphere to independent standard normals when $k = o(n)$.  Independently, Jiang and Ma \cite{jiang2} and Stewart \cite{stewart} (extending on work in \cite{jiang}) have further strengthened this result considerably to show that if $p_n \ge 1$, $q_n \ge 1$ and $p_n q_n /n \to 0$, then the total variation distance between the joint distribution of the entries of the upper left $p_n \times q_n$ block of an $n \times n$ Haar-distributed orthogonal matrix and $p_nq_n$ independent standard normals converges to 0.  Moreover, they show that this result is sharp.  In \cite{jiang2}, corresponding results for other distance metrics are proven as well.  

Matrix entries of random permutation matrices have also been studied.  Hoeffding's combinatorial central limit theorem \cite{hoeffding} gives a result on linear combinations of random permutation matrix entries.  Let $A_n$ be a sequence of $n \times n$ real matrices such that $Tr(A_n A_n^T) = n$.  Hoeffding shows that under certain conditions on $A_n$, if $P$ is a random $n \times n$ permutation matrix, then $Tr(A_nP)$ converges weakly to a standard normal random variable as $n \to \infty$.  

In this article, we study the distribution of entries of a random permutation matrix in a basis free way.  Individual entries of a permutation matrix are either 0 or 1, but if we conjugate by a random orthogonal matrix, we can look at the distribution of entries under a ``randomized basis.''   To formalize this, let $P$ be an $n \times n$ permutation matrix drawn uniformly from the symmetric group $\mathfrak{S}_n$ and let $M$ be a Haar-distributed $n \times n$ orthogonal matrix independent of $P$.  (With little confusion, we will often conflate a permutation matrix and its corresponding permutation in $\mathfrak{S}_n$).  Let $A_n$ be a sequence of $n \times n$ real matrices such that $Tr(A_n A_n^T) = n$.  Then it is shown in Theorem \ref{main result} below that under certain conditions on $A_n$, \begin{equation} Tr(A_n M P M^T) \buildrel d \over \to sY+Z \end{equation} as $n \to \infty$ where $Y$ and $Z$ are independent random variables such that $Y$ is Poisson distributed, $Z$ is normally distributed, and $s$ is a constant.  Whereas Hoeffding's combinatorial CLT shows convergence for linear combinations of random permutation matrix entries, Theorem \ref{main result} can be interpreted as a distributional convergence for linear combinations of random permutation matrix entries under a ``randomized basis.''

We briefly review some other related results.  Let $A$ be an $n \times n$ real (nonrandom) matrix such that $Tr(A A^T) = n$.  D'Aristotile, Diaconis, and Newman \cite{diaconisnewman} have shown using characteristic function methods that if $M$ is an $n \times n$ random orthogonal matrix, then $Tr(A M)$ converges in distribution to a standard normal random variable as $n \to \infty$ uniformly in $A$.  Meckes \cite{meckes} gives quantitative bounds and shows that the total variation distance between $Tr(A M)$ and a standard normal random variable is bounded by $\frac{2 \sqrt{3}}{n-1}$ and that this rate is sharp up to a constant.  In \cite{chatterjee}, Chatterjee and Meckes obtain bounds on the Wasserstein distance between the multivariate distributions $(Tr(A_1M), Tr(A_2 M),...,Tr(A_k M))$ and a Gaussian random vector.  Their techniques involve generalizing Stein's method of exchangeable pairs.

\section{Limiting distribution of a single entry} \label{single entry}

In this section, we start by showing that a single scaled matrix entry of $MP M^T$ converges in distribution to a standard normal distribution.  The next section will consider the distribution of linear combinations of multiple matrix entries.  

Note that the distribution of $MP M^T$ where $P$ is random can be thought of as a mixture of  distributions of $MP M^T$ where $P$ is a fixed permutation matrix.  Thus, we first assume $P$ is fixed.  If $P$ is similar to another permutation matrix $Q$ via an orthogonal transformation $V$, i.e. $P = VQV^T$, then $MP M^T = (MV)Q(MV)^T \buildrel d \over = MQM^T$.  Thus, the distribution of $MPM^T$ only depends on the cycle type of the permutation matrix $P$.     

First, we consider the case where the permutation consists of the single $n$-cycle $(1,2,...,n)$.  Let $C_n$ be the corresponding permutation matrix.  Then \begin{equation} (M C_n M^T)_{ab} = \sum_{i=1}^n M_{a,i}M_{b,i+1} \end{equation}  (where abusing notation slightly, $M_{b,n+1}:= M_{b,1}$) By symmetry, it is clear that $\mathbb{E}[M_{ij}^2] = 1/n$.  Moreover, $\displaystyle{\sqrt{n}(M_{a,i}, M_{b, i+1}) \buildrel d \over \rightarrow (X_1, X_2)}$ where $X_1$ and $X_2$ are i.i.d $\mathcal{N}(0,1)$.  (The convergence is actually in total variation, see e.g. \cite{diaconisfreedman, jiang}).  Therefore, $\sum\limits_{i=1}^n M_{a,i}M_{b,i+1}$ should have variance on the order of $1/n$ and it is natural to consider the scaled random variable $\displaystyle{\sqrt{n} (M C_n M^T)_{ab}}$.  

To obtain the limiting distribution of $\displaystyle{\sqrt{n} (M C_n M^T)_{ab}}$ (as well as the limiting distribution involving multiple entries derived later on), the key tool we will apply is a martingale central limit theorem for dependent random variables proved by McLeish \cite{mcleish}.  (Other approaches such as the moment method \cite{dey} have also been explored.)
 
Specifically, we will use Corollary (2.13) of \cite{mcleish} to prove our result.  Let $\{X_{n,i}; 1 \le i \le n\}$ be a triangular array of random variables.  Define $\displaystyle{\sigma_{n,i}^2 = \mathbb{E}[X_{n,i}^2]}$, $\displaystyle{s_n^2 = \sum_{i=1}^n \sigma_{n,i}^2}$, and $\displaystyle{S_n = \sum_{i=1}^{n} X_{n,i}}$.  Let ${\mathcal{F}}_{n,i} = \sigma(X_{n,1}, X_{n,2},...,X_{n,i})$.  We say that $X_{n,i}$ is a martingale difference array if $\mathbb{E}(X_{n,i} | {\mathcal{F}}_{n,i-1}) = 0$.  Then we have the following:

\begin{theorem}[McLeish]\label{mcleish} Suppose $X_{n,i}$ is a martingale difference array normalized by its variance $s_n^2$, satisfying the Lindeberg condition $\displaystyle{\lim_{n \to \infty} \sum_{i=1}^n \int_{|X_{n,i}| > \epsilon} X_{n,i}^2 dP = 0}$ for all $\varepsilon > 0$ and the condition \begin{equation}\label{mcleishcondition} \limsup_{n \to \infty} \sum_{i \neq j} \mathbb{E}[X_{n,i}^2 X_{n,j}^2] \le 1\end{equation}  Then $S_n \buildrel d \over \to \mathcal{N}(0,1)$.      

\end{theorem}

 For a book-length treatment on martingale limit theorems, see \cite{heyde}.  The following lemma will also be useful in the sequel:

\begin{lemma}\label{moments}

Let the random variable $X_n = n^k M_1^{k_1}...M_{r}^{k_r}$ where $k_1+...+k_r = 2k$ and $M_i$ are distinct entries of an $n \times n$ Haar distributed orthogonal matrix.  Define $Y = Z_1^{k_1}...Z_{r}^{k_r}$ where $Z_i$ are i.i.d. standard normals.  Then $X_n \buildrel L^1 \over \to Y$.   

\end{lemma}

\begin{proof}
By \cite{jiang}, the joint distribution of any $k$ fixed entries (scaled by $\sqrt{n}$) of a Haar distributed orthogonal matrix converges in total variation to the joint distribution of $k$ i.i.d. standard normals.  In particular, by the continuous mapping theorem, we see that $X_n \buildrel d \over \to Y$.   To get $L^1$ convergence, we need uniform integrability, which is implied by $\limsup\limits_{n \to \infty} \mathbb{E}[ X_n^2] < \infty.$  By the Cauchy-Schwarz inequality, it is sufficient to show that $\limsup\limits_{n \to \infty} \mathbb{E}[ n^{2k} M_1^{4k}] < \infty.$  But $M_1^2$ has a Beta distribution with parameters $(1/2, (n-1)/2)$ (see e.g. \cite[pp.~145]{petz}) and hence $\displaystyle{\mathbb{E}[ n^{2k} M_1^{4k}] = O(1)}$.  

\end{proof}

\begin{remark}\label{collins}
Lemma \ref{moments} gives rather crude asymptotics for the expectation of products of Haar distributed matrix entries.  Collins and Sniady \cite{collinssniady} have developed machinery using Weingarten functions to compute the integrals of polynomial functions of Haar distributed orthogonal, unitary, and symplectic entries for any dimension $n$.  (Please refer to \cite{collinssniady} for the definition of Weingarten functions.  We will only require the asymptotics of Weingarten functions mentioned below.)

For the orthogonal group, Collins and Sniady's formula says that the integral over Haar measure $\displaystyle{\int_{M \in O(n)} M_{i_1 j_1}...M_{i_{2r} j_{2r}} dM }$ is the sum of Weingarten functions $Wg(\mathfrak{m}, \mathfrak{n})$ over pair partitions $\mathfrak{m}$ and $\mathfrak{n}$ of the set $\{1,...,2r\}$ such that $\displaystyle{i_{\mathfrak{m}(2k-1)} = i_{\mathfrak{m}(2k)}}$ and $\displaystyle{j_{\mathfrak{n}(2k-1)} = j_{\mathfrak{n}(2k)} }$ where $\displaystyle{\mathfrak{m} = \{\{\mathfrak{m}(1),\mathfrak{m}(2)\},...,\{\mathfrak{m}(2r-1),\mathfrak{m}(2r)\}\}}$ and $\mathfrak{n} = \{\{\mathfrak{n}(1),\mathfrak{n}(2)\},...,\{\mathfrak{n}(2r-1),\mathfrak{n}(2r)\}\}$.  In other words, $\mathfrak{m}$ represents pairing up matrix entries in the same row and $\mathfrak{n}$ represents pairing up entries in the same column.  In particular, the expectation is non-zero only if there are an even number of entries in each row and column.   

In \cite{collinsmatsumoto}, Collins and Matsumoto obtain a simpler formula for the orthogonal Weingarten function that significantly reduces the complexity involved in the computation of Weingarten formulas and show that the integrals are always rational functions of the dimension $n$.  They also provide asymptotics for the Weingarten function: 
\begin{equation}\label{weingartenFormula} Wg(\mathfrak{m}, \mathfrak{n}) = O(1/n^{-2r + l(\mathfrak{m}, \mathfrak{n})}) \end{equation}
Here, the length $l(\mathfrak{m}, \mathfrak{n})$ is defined as follows: consider the graph with vertex set ${1, 2,...,2r}$ and edge set that consists of $\{\mathfrak{m}(2k-1), \mathfrak{m}(2k)\}, \{\mathfrak{n}(2k-1), \mathfrak{n}(2k)\}$ where $1 \le k \le r$.  Then $l(\mathfrak{m}, \mathfrak{n})$ is the number of components in this graph.  Thus, finding the asymptotics of the integral $\displaystyle{\int_{M \in O(n)} M_{i_1 j_1}...M_{i_{2r} j_{2r}} dM }$ is equivalent to finding the (admissible) pairings $\mathfrak{m}$ and $\mathfrak{n}$ that maximize the number of components in the graph.  See
\cite{collinsmatsumoto} also for a useful table of values for the orthogonal Weingarten functions of small dimension. 

\end{remark}

For our purposes, it will be simpler to use the stronger Lyapunov's condition in place of the Lindeberg condition in McLeish's theorem.  For $\delta > 0$, Lyapunov's condition says that \begin{equation} \label{Lyapunov} \lim_{n \to \infty} \frac{1}{s_n^{2+\delta}} \sum_{i=1}^n \mathbb{E}\Big[|X_{n,i}|^{2+\delta} \Big] = 0. \end{equation}  

We apply this now to show

\begin{lemma} \[\sqrt{n} \sum_{i=1}^n M_{a,i}M_{b,i+1} \buildrel d \over \rightarrow \mathcal{N}(0,1)\] \end{lemma}   
\begin{proof}

Let $X_{n,i} = \sqrt{n} M_{a,i}M_{b,i+1}$.  By sign symmetry, $\mathbb{E}[X_{n,i}] = 0$.  Note that the random vector $(X_{n,1}, X_{n,2},...,X_{n,i-1})$ is only a function of $(M_{a,1}, M_{a,2},...,M_{a,i-1},$ \allowbreak $M_{a,i}, M_{b,2},...,M_{b,i})$.  Since $M_{b,i+1}$ does not share a column with any of these matrix entries, its sign can be reversed independently while preserving Haar measure.  Thus, \[(X_{n,1}, X_{n,2},...,X_{n,i-1}, M_{b,i+1}) \buildrel d \over = (X_{n,1}, X_{n,2},...,X_{n,i-1}, -M_{b,i+1})\] and therefore \[\mathbb{E}(X_{n,i} | {\mathcal{F}}_{n,i-1}) = \sqrt{n}\mathbb{E}(M_{a,i}M_{b,i+1} | \sigma(X_{n,1}, X_{n,2},...,X_{n,i-1})) = 0.\]  
This shows that $X_{n,i}$ is a martingale difference array.  Let us now verify the conditions in McLeish's theorem.  

The variance $\sigma_{n,i}^2$ is $\displaystyle{\frac{1}{n+2}}$ when $a = b$ and $\displaystyle{\frac{n+1}{(n-1)(n+2)}}$ when $a \neq b$ (see e.g. \cite[chapt. 4]{petz}).  Thus, $\displaystyle{\lim_{n \to \infty} \sum_{i=1}^n \sigma_{n,i}^2 = 1}$ and the array is already normalized.  

By the row invariance of Haar measure, all the random variables $X_{n,i}$ are identically distributed.  Thus, Lyapunov's condition for $\delta = 2$ simply states that $\displaystyle{\lim_{n \to \infty} n \mathbb{E}[ X_{n,i}^4] = 0}$.  By Lemma \ref{moments}, $\displaystyle{n^2 \mathbb{E}[ X_{n,i}^4] = n^4 \mathbb{E}[M_{a,i}^4 M_{b,i+1}^4] = O(1) }$ and Lyapunov's condition is satisfied.    

It remains to verify condition \eqref{mcleishcondition}.  We split this into two sums: \begin{equation} \sum_{i \neq j} \mathbb{E} [X_{n,i}^2 X_{n,j}^2] = \sum_{|i-j| > 1} \mathbb{E} [X_{n,i}^2 X_{n,j}^2] + \sum_{|i-j| = 1} \mathbb{E} [X_{n,i}^2 X_{n,j}^2] \end{equation}  Note that the random variables $X_{n,i}^2 X_{n,j}^2$ such that $|i-j| > 1$ are all identically distributed, and similarly for $|i-j| = 1$.  For the first sum, we have \[\lim_{n \to \infty} \sum_{|i-j| > 1} \mathbb{E} [X_{n,i}^2 X_{n,j}^2] = \lim_{n \to \infty} \mathbb{E} [n^2 X_{n,1}^2 X_{n,3}^2] = \lim_{n \to \infty} \mathbb{E}[n^4 M_{a,1}^2 M_{b,2}^2 M_{a,3}^2 M_{b,4}^2] = 1.\]  The asymptotic of the second sum is \[\sum_{|i-j| = 1} \mathbb{E} [X_{n,i}^2 X_{n,j}^2] = O(1/n)\] and therefore \[\limsup_{n \to \infty} \sum_{i \neq j} \mathbb{E} [X_{n,i}^2 X_{n,j}^2] = 1.\] 
\end{proof}

We've shown that if $C_n$ is the matrix corresponding to the $n$-cycle $(1,2,3,...,n)$ then each matrix entry $(M C_n M^T)_{ab}$ scaled by $\sqrt{n}$ converges to the standard normal distribution.  It turns out that $C_n$ is in fact a fairly generic permutation in $\mathfrak{S}_n$.  The following lemma makes this precise.  

\begin{lemma}\label{Q_n} Let $M$ be an $n \times n$ Haar distributed orthogonal matrix and $Z$ be a $\mathcal{N}(0,1)$ distributed random variable.  Then \[\lim_{n \to \infty} \max_{Q_n} |\mathbb{P}\{\sqrt{n}(M Q_n M^T)_{ab} \le x\} - \mathbb{P}\{Z \le x\}| = 0 \] 
where the maximum is taken over the set of all $n \times n$ permutation matrices, $Q_n$, such that the number of cycles in the corresponding permutation is at most $2 \log n$.

\end{lemma}   

\begin{proof}
Let $\delta > 0$ and let $Q_n$ be a permutation matrix whose corresponding permutation $\sigma \in \mathfrak{S}_n$ has $k_n \le 2 \log n$ cycles.  For ease of notation, we set $X_n = \sqrt{n}(M Q_n M^T)_{ab}$ and $Y_n = \sqrt{n}(M C_n M^T)_{ab}$.  
Since the distribution of $X_n$ only depends on the cycle type of $\sigma$, we can assume  $\sigma = (1, 2, 3, ..., n_1) (n_1 + 1, n_1 + 2,...,n_2)...(n_{k_n-1} + 1, n_{k_n-1} + 2,...,n)$ when decomposed into cycles.  Then \[X_n - Y_n = \sqrt{n} \sum_{i=1}^n M_{a,i}M_{b,\sigma(i)} - \sqrt{n} \sum_{i=1}^n M_{a,i}M_{b,i+1} = \sum_{j = 1}^{2 k_n} W_j \] where  $W_j \buildrel d \over = \sqrt{n} M_{a,c_j} M_{b,d_j}$ for some indices $c_j$ and $d_j$.  (Note that all terms in the difference cancel except those involving $M_{a, n_j}$ for indices $n_j$ where $\sigma$ and the $n$-cycle $(1,...,n)$ differ.)  For sufficiently large $n$, $\displaystyle{\mathbb{E}[W_j^2] < C/n}$ for some constant $C$ and by Cauchy-Schwarz, \[\mathbb{E}[(X_n - Y_n)^2] < 4 (k_n)^2 (C/n). \]

By Chebyshev's inequality, \[ \mathbb{P}\{|X_n - Y_n| > \varepsilon\} \le \varepsilon^{-2} \mathbb{E}[(X_n - Y_n)^2] \le \varepsilon^{-2} \frac{16 C \log^2 n}{n}.\]  We also have the two inequalities
\begin{align*}
&\mathbb{P}\{ X_n \le x\} \le \mathbb{P}\{ Y_n \le x + \varepsilon \} + \mathbb{P}\{ |X_n - Y_n| > \varepsilon \} \\
&\mathbb{P}\{ Y_n \le x - \varepsilon\} \le \mathbb{P}\{ X_n \le x\} + \mathbb{P}\{ |X_n - Y_n| > \varepsilon\}
\end{align*}
Putting them together, we have \[\mathbb{P}\{ Y_n \le x - \varepsilon\} - \mathbb{P}\{ |X_n - Y_n| > \varepsilon\} \le \mathbb{P}\{ X_n \le x\} \le \mathbb{P}\{ Y_n \le x + \varepsilon\} + \mathbb{P}\{ |X_n - Y_n| > \varepsilon\} \]
Thus, 
\begin{align*}
|\mathbb{P}\{ X_n \le x\} - \mathbb{P}\{ Y_n \le x\} | &\le \mathbb{P}\{ Y_n \le x + \varepsilon\} - \mathbb{P}\{ Y_n \le x - \varepsilon\} + 2 \mathbb{P}\{ |X_n - Y_n| > \varepsilon\} \\
&\le \mathbb{P}\{ Y_n \le x + \varepsilon\} - \mathbb{P}\{ Y_n \le x - \varepsilon\} + \varepsilon^{-2} \frac{32C \log^2 n}{n}
\end{align*}
By the triangle inequality, 
\begin{multline*}
\mathbb{P}\{ Y_n \le x + \varepsilon\} - \mathbb{P}\{ Y_n \le x - \varepsilon\} \le |\mathbb{P}\{ Y_n \le x + \varepsilon\} - \mathbb{P}\{ Z \le x + \varepsilon\}| \\ 
+ \mathbb{P}\{ Z \le x + \varepsilon\} - \mathbb{P}\{ Z \le x - \varepsilon\} + |\mathbb{P}\{ Y_n \le x - \varepsilon\} - \mathbb{P}\{ Z \le x - \varepsilon\}|
\end{multline*}
Also, \[ |\mathbb{P}\{X_n \le x\} - \mathbb{P}\{Z \le x\}| \le |\mathbb{P}\{X_n \le x\} - \mathbb{P}\{Y_n \le x\}| + |\mathbb{P}\{Y_n \le x\} - \mathbb{P}\{Z \le x\} |.\]  Since $Y_n$ converges weakly to $Z$, $|\mathbb{P}\{Y_n \le x\} - \mathbb{P}\{ Z \le x \}| < \delta/2$ for sufficiently large $n$.  Choose small enough $\varepsilon$ such that $\mathbb{P}\{ Z \le x + \varepsilon\} - \mathbb{P}\{ Z \le x - \varepsilon\} < \delta/10$.  For large enough $n$, the three quantities $| \mathbb{P}\{ Y_n \le x + \varepsilon\} - \mathbb{P}\{ Z \le x + \varepsilon\} |$, $|\mathbb{P}\{ Y_n \le x - \varepsilon\} - \mathbb{P}\{ Z \le x - \varepsilon\} |$, and $\displaystyle{\varepsilon^{-2} \frac{32C \log^2 n}{n} }$ are all smaller than $\delta/10$.  Then $|\mathbb{P}\{ X_n \le x\} - \mathbb{P}\{ Y_n \le x \}| < \delta/2$.  Since $Q_n$ was arbitrary, this proves the lemma.  
\end{proof}

Let $K_n(\sigma)$ denote the number of cycles in a random permutation $\sigma \in \mathfrak{S}_n$.  The classical Goncharov theorem (see e.g. \cite[p.~116]{durrett}) states that $\displaystyle{\frac{K_n - \log n}{\sqrt{\log n}} \buildrel d \over \rightarrow \mathcal{N}(0,1)}$.  This implies in particular that $\mathbb{P}\{K_n > 2 \log n\} \rightarrow 0$.  Now we are ready to prove: 
\begin{theorem}\label{single}
Let $M$ be a Haar distributed orthogonal $n \times n$ matrix and $P$ a random $n \times n$ permutation matrix independent of $M$.  Then  
\[\sqrt{n}(M P M^T)_{ab} \buildrel d \over \rightarrow \mathcal{N}(0,1) \]
\end{theorem}

\begin{proof}
Averaging over all permutation matrices, \[\mathbb{P}\{ \sqrt{n}(M P M^T)_{ab} \le x\} = \frac{1}{n!} \sum_{\sigma \in \mathfrak{S}_n} \mathbb{P}\{ \sqrt{n}(M P_\sigma M^T)_{ab} \le x\} \]  
For any $\delta > 0$ and sufficiently large $n$, Lemma \ref{Q_n} gives 
\begin{align*}
&|\mathbb{P}\{ \sqrt{n}(M P M^T)_{ab} \le x\} - \mathbb{P}\{ Z \le x\}| \\
\le & \frac{1}{n!} \sum_{\sigma \in \mathfrak{S}_n} |\mathbb{P}\{ \sqrt{n}(M P_\sigma M^T)_{ab} \le x\} - \mathbb{P}\{ Z \le x\} | \\
\le & \delta/2 + \mathbb{P}\{ K_n > 2 \log n\} \\
< & \delta  
\end{align*}

\end{proof}

\section{Limiting distribution of linear combination of multiple entries}

We can generalize Theorem \ref{single} to obtain the limiting distribution of arbitrary linear combinations of the matrix entries of $MPM^T$, i.e. for a sequence of $n \times n$ real matrices $A_n$ we will determine the asymptotic behavior of $Tr(A_n M P M^T)$.  Theorem \ref{single} says that this trace converges to a standard normal when the $(a,b)$ entry of $A_n$ is $\sqrt{n}$ and all the other entries are 0.  

We will now consider more general coefficient matrices $A_n$ normalized so that $Tr(A_n A_n^T) = n$.  For convenience, the subscript will often be dropped. 
\begin{remark}
Recall that the limiting normality of $Tr(A M)$ was proven in \cite{diaconisnewman}.  A simple but key observation in the proof is that we can reduce to the case of diagonal $A$ by using the singular value decomposition.  Writing $A = UDV^T$, we get $Tr(AM) = Tr(UDV^TM) = Tr(DV^TMU)$.  This is equal in distribution to $Tr(DM)$ by left and right Haar invariance.  Unfortunately, we cannot reduce to the case of diagonal matrices in our situation since \begin{align*} Tr(A MP M^T) = Tr(UDV^T MP M^T) &= Tr(DV^T MP M^TU) \\ &= Tr(D (V^TM) P (U^T M)^T) \end{align*} and in general $U \neq V$.  If $A$ is normal, e.g. symmetric or orthogonal, then $U = V$ and it would be possible to reduce to diagonal $A$.  However, we will not make this assumption.  
\end{remark}

As before, let $C_n$ denote the $n \times n$ permutation matrix corresponding to the $n$-cycle $(1,2,...,n)$.  Then 
\begin{equation}Tr(A M C_n M^T) = \sum_{i = 1}^n \sum_{j=1}^n A_{ij}(M C_n M^T)_{ji} = \sum_{k=1}^n X_{n,k} \end{equation}
where 
\begin{equation}\label{Xnk} X_{n,k} = \sum_{i=1}^n \sum_{j=1}^n A_{ij} M_{j,k}M_{i,k+1}. 
\end{equation}  It is easy to see that $X_{n,k}$ is a martingale difference array.  Squaring $X_{n,k}$ and expanding, we get

\begin{align*}
X_{n,k}^2 &= \sum_{i=1}^n \sum_{j=1}^n A_{ij}^2 M_{j,k}^2M_{i,k+1}^2 \\
&+ \sum_{i=1}^n \sum_{j=1}^n \sum_{m \neq j} A_{ij} A_{im} M_{j,k}M_{m,k}M_{i,k+1}^2 + \sum_{i=1}^n \sum_{j=1}^n \sum_{l \neq i} A_{ij}A_{lj} M_{j,k}^2M_{i,k+1}M_{l,k+1} \\
&+ \sum_{i=1}^n \sum_{j=1}^n \sum_{l \neq i} \sum_{m \neq j} A_{ij}A_{lm} M_{j,k}M_{i,k+1}M_{m,k}M_{l,k+1}
\end{align*}
 
\noindent By \eqref{weingartenFormula} and the discussion in Remark \ref{collins}, the mixed moments are 

\begin{align*} \mathbb{E} [M_{j,k}^2M_{i,k+1}^2] &= 1/n^2 + O(1/n^3) \\ \mathbb{E}[M_{j,k}^2M_{i,k+1}M_{l,k+1}] &= 0  \\
\mathbb{E}[M_{j,k}M_{i,k+1}M_{m,k}M_{l,k+1}] &= \left\{
\begin{array}{lr}
\frac{-1}{(n-1)n(n+2)}  & \text{ if } j = i, m = l \text{ or } j = l, i = m \cr
0 & \text{ otherwise }
\end{array}
\right.  
\end{align*} 
 
\noindent Therefore the variance is \begin{equation}\label{variance} \mathbb{E}[ X_{n,k}^2] = \left(\frac{1}{n^2} + O\left(\frac{1}{n^3}\right)\right)\sum_{i=1}^n \sum_{j=1}^n A_{ij}^2  - \frac{1}{(n-1)n(n+2)} \sum_{i \neq j} (A_{ij}A_{ji} + A_{ii}A_{jj}) \end{equation}  
Recall we have the constraint $\sum A_{ij}^2 = n$.  It is clear that $\sum\limits_{i \neq j} A_{ij}A_{ji}$ is maximized when $A_{ij} = A_{ji}$ and minimized when $A_{ij} = -A_{ji}$.  In any case, the magnitude of the sum is at most $n$.  By the method of Lagrange multipliers, $\sum\limits_{i \neq j} A_{ii}A_{jj}$ achieves a maximum value of $n^2 - n$ when $A$ is the identity matrix and a minimum value of $-n$ when $\sum A_{ii} = 0$.  Without any further restrictions, the sum $\sum\limits_{i \neq j} A_{ii}A_{jj}$ can fluctuate wildly.  For there to be any hope of a limiting distribution for $Tr(AMC_nM^T)$, we need to have $\lim\limits_{n \to \infty} \sum\limits_{i \neq j}\frac{A_{ii}A_{jj}}{n^2} = c$ for some constant $0 \le c \le 1$.  Then we have the following:

\begin{theorem}\label{main} Let $A_n$ be a sequence of $n \times n$ matrices such that $Tr(A^TA) = n$.  As usual, $M$ denotes an $n \times n$ Haar distributed orthogonal matrix and $C_n$ the permutation matrix $(1,2,...,n)$.  If 
\begin{equation*} \lim_{n \to \infty} \sum_{i \neq j}\frac{A_{ii}A_{jj}}{n^2} = c, \end{equation*} then
\begin{equation*}Tr(AMC_nM^T) \buildrel d \over \rightarrow \mathcal{N}(0, 1-c). \end{equation*} 
\end{theorem}     

\begin{proof}
From equation \eqref{variance}, we see that the limiting variance of $Tr(AMC_nM^T)$ is given by \begin{equation}\lim_{n \to \infty} \sum_{k=1}^n \sigma_{n,k}^2 = \lim_{n \to \infty}  n \mathbb{E} X_{n,1}^2 = 1 - c \end{equation}  

Next, we show that Lyapunov's condition for the martingale difference array $X_{n,k}$ holds with $\delta = 2$, i.e. $\displaystyle{\lim_{n \to \infty} n \mathbb{E}[X_{n,k}^4] = 0}$.  

When expanded out, $X_{n,k}^4$ contains terms of the form $\displaystyle{\prod_{p=1}^{4} A_{i_p j_p} M_{j_p,k}M_{i_p,k+1}}$.  Recall the only terms with non-zero expectation are those with an even number of matrix entries $M_{ij}$ in each row and column.  Working through the possibilities, we get (using the shorthand $i \neq j \neq l \neq m$ to mean $i,j,l,m$ all distinct and $\lesssim$ to denote inequality up to an absolute constant as $n \to \infty$)
\begin{align*}
\mathbb{E}[&X_{n,k}^4] \lesssim \sum_{i \neq j \neq l \neq m} (A_{ii}A_{jj}A_{ll}A_{mm}+ A_{ij}A_{ji}A_{ll}A_{mm}+ A_{ij}A_{ji}A_{lm}A_{ml} \\
&+ A_{ij}A_{jl}A_{li}A_{mm} + A_{ij}A_{jl}A_{lm}A_{mi} ) \mathbb{E}[M_{i,k}M_{i,k+1} M_{j,k}M_{j,k+1}M_{l,k}M_{l,k+1} M_{m,k} M_{m,k+1}] \\
&+\sum_{\substack{i \neq j \neq l \\i \neq j \neq m}} (A_{ii}A_{jj}A_{ml}^2 +A_{ii}A_{jl}A_{mj}A_{ml}+ A_{ij}A_{ji}A_{ml}^2 + A_{ij}A_{jl}A_{mi}A_{ml} \\
& + A_{il}A_{jl}A_{mi}A_{mj} )\mathbb{E}[M_{i,k}M_{i,k+1} M_{j,k}M_{j,k+1}M_{l,k}^2M_{m,k+1}^2] \\
&+\sum_{i \neq j \neq l} (A_{ii}A_{jj}A_{il}^2 + A_{ii}A_{ij}A_{il}A_{jl} + A_{ji}A_{ij}A_{il}^2  )\mathbb{E}[M_{i,k}M_{i,k+1}^3 M_{j,k}M_{j,k+1}M_{l,k}^2] \\
&+\sum_{i \neq j} (A_{ii}A_{jj}^3 + A_{ij}A_{ji}A_{jj}^2)\mathbb{E}[M_{i,k}M_{i,k+1} M_{j,k}^3 M_{j,k+1}^3] \\
&+\sum_{i \neq j} (A_{ii}A_{jj}A_{ij}^2 + A_{ji}A_{ij}^3)\mathbb{E}[M_{i,k}M_{i,k+1}^3 M_{j,k}^3 M_{j,k+1}] \\
&+\sum_{\substack{i \neq l \\ j \neq m}} (A_{ji}^2 A_{ml}^2 + A_{ji}A_{mi}A_{jl}A_{ml})\mathbb{E}[M_{i,k}^2M_{j,k+1}^2 M_{l,k}^2 M_{m,k+1}^2] \\
&+\sum_{\substack{j \\ i \neq l}} (A_{ji}^2 A_{jl}^2)\mathbb{E}[M_{i,k}^2M_{j,k+1}^4 M_{l,k}^2 ] \\
&+\sum_{i, j} (A_{ji}^4 )\mathbb{E}[M_{i,k}^4 M_{j,k+1}^4 ] 
\end{align*}

\noindent By \eqref{weingartenFormula}, it is easy to compute 
\begin{align*}
\mathbb{E}[M_{i,k}M_{i,k+1} M_{j,k}M_{j,k+1}M_{l,k}M_{l,k+1} M_{m,k} M_{m,k+1}] &= O(1/n^6) \\
\mathbb{E}[M_{i,k}M_{i,k+1} M_{j,k}M_{j,k+1}M_{l,k}^2M_{m,k+1}^2] &= O(1/n^5)\\
\mathbb{E}[M_{i,k}M_{i,k+1}^3 M_{j,k}M_{j,k+1}M_{l,k}^2] &= O(1/n^5) \\
\mathbb{E}[M_{i,k}M_{i,k+1} M_{j,k}^3 M_{j,k+1}^3] &= O(1/n^5) \\
\mathbb{E}[M_{i,k}M_{i,k+1}^3 M_{j,k}^3 M_{j,k+1}] &= O(1/n^5) \\
\mathbb{E}[M_{i,k}^2M_{j,k+1}^2 M_{l,k}^2M_{m,k+1}^2] &= O(1/n^4) \\
\mathbb{E}[M_{i,k}^2M_{j,k+1}^4 M_{l,k}^2 ] &= O(1/n^4) \\
\mathbb{E}[M_{i,k}^4 M_{j,k+1}^4 ] &= O(1/n^4)
\end{align*}

The sums involving the matrix entries of $A$ are straightforward to bound.  For example, we have \[
\sum_{\substack{i \neq l \\ j \neq m}} A_{ij}A_{im}A_{lj}A_{lm} \le \sum_{i,j,l,m} (A_{ij}^2A_{lm}^2+ A_{im}^2 A_{lj}^2)/2 \le n^2 \]

Going through all the terms, we see that $\displaystyle{\mathbb{E}[X_{n,k}^4] = O(1/n^2)}$ and Lyapunov's condition is satisfied.  The last condition of McLeish's Theorem to verify is that \begin{equation} \limsup_{n \to \infty} \sum_{k \neq l} \mathbb{E} [X_{n,k}^2 X_{n,l}^2] = (1-c)^2\end{equation}  (Note we didn't normalize the array $X_{n,k}$.)  By Cauchy-Schwarz, \begin{equation} \mathbb{E}[X_{n,k}^2 X_{n,l}^2] \le \mathbb{E}[X_{n,k}^4] = O(1/n^2) \end{equation}  Thus, $\displaystyle{ \lim_{n \to \infty} \sum_{|l-k| = 1} \mathbb{E}[X_{n,k}^2 X_{n,l}^2] = 0}$ and therefore we can assume that $|l-k| > 1$.  Written out explicitly, we have \[ X_{n,k}^2 X_{n,l}^2 = \left(\sum_{a=1}^n \sum_{b=1}^n A_{ab} M_{b,k}M_{a,k+1}\right)^2 \left(\sum_{c=1}^n \sum_{d=1}^n A_{cd} M_{d,l}M_{c,l+1}\right)^2 \]
Expanding out and going through all the possibilities, the expectation $\mathbb{E}[X_{n,k}^2 X_{n,l}^2]$ contains the following types of terms:
\begin{enumerate}
\item $\displaystyle{\begin{aligned}[t] &\sum_{\substack{a \neq b \cr c \neq d}} (A_{aa}A_{bb}A_{cc}A_{dd} + A_{ab}A_{ba}A_{cd}A_{dc}
+ A_{aa}A_{bb}A_{cd}A_{dc}) \mathbb{E}[ M_{a,k} M_{a,k+1}M_{b,k}M_{b,k+1} \\& M_{c,l}M_{c,l+1}M_{d,l}M_{d,l+1}] \end{aligned} }$
\item $\displaystyle{\sum_{\substack{a \neq b \\ c \neq d}} (A_{ca}A_{db}A_{ca}A_{db} + A_{ca}A_{db}A_{da}A_{cb})\mathbb{E}[ M_{a,k} M_{c,k+1}M_{b,k}M_{d,k+1} M_{a,l}M_{c,l+1}M_{b,l}M_{d,l+1}]}$ 
\item $\displaystyle{\begin{aligned}[t] & \sum_{\substack{a, d \neq b, c}} (A_{aa}A_{cb}A_{cb}A_{dd} + A_{aa}A_{cb}A_{cd}A_{db} + A_{ca}A_{ab}A_{cd}A_{db})\mathbb{E}[ M_{a,k} M_{a,k+1}M_{b,k}M_{c,k+1} \\& M_{b,l}M_{c,l+1}M_{d,l}M_{d,l+1}] \end{aligned}}$
\item $\displaystyle{\begin{aligned}[t]
&\sum_{\substack{a, d \neq b, c}} (A_{ca}A_{db}A_{ba}A_{dc} + A_{ca}A_{db}A_{da}A_{bc} + A_{da}A_{cb}A_{ba}A_{dc} \\
&+ A_{da}A_{cb}A_{da}A_{bc})\mathbb{E}[ M_{a,k} M_{c,k+1}M_{b,k}M_{d,k+1} M_{a,l}M_{b,l+1}M_{c,l}M_{d,l+1}]
\end{aligned}}$ 
\item $\displaystyle{\sum_{\substack{d \\ a \neq b \neq c}} (A_{aa}A_{cb}A_{db}A_{dc} + A_{ca}A_{ab}A_{db}A_{dc})\mathbb{E}[ M_{a,k} M_{a,k+1}M_{b,k}M_{c,k+1} M_{b,l}M_{c,l}M_{d,l+1}^2]}$ 
\item $\displaystyle{\sum_{\substack{c, d \\ a \neq b}} (A_{aa}A_{bb}A_{dc}^2 + A_{ab}A_{ba}A_{dc}^2)\mathbb{E}[ M_{a,k} M_{a,k+1}M_{b,k}M_{b,k+1} M_{c,l}^2M_{d,l+1}^2]}$ 
\item $\displaystyle{\sum_{\substack{c,d \\ a \neq b}} (A_{ca}A_{cb}A_{da}A_{db} )\mathbb{E}[ M_{a,k}M_{b,k} M_{c,k+1}^2 M_{a,l}M_{b,l} M_{d,l+1}^2]}$ 
\item $\displaystyle{\sum_{a,b,c,d} (A_{ba}^2 A_{dc}^2 )\mathbb{E}[ M_{a,k}^2M_{b,k+1}^2 M_{c,l}^2 M_{d,l+1}^2]}$ 
\end{enumerate}

The orthogonal integrals (using \eqref{weingartenFormula} ) are given by
\begin{align*}
\mathbb{E}[M_{a,k} M_{a,k+1}M_{b,k}M_{b,k+1} M_{c,l}M_{c,l+1}M_{d,l}M_{d,l+1}] &= O(1/n^6) \\
\mathbb{E}[ M_{a,k} M_{c,k+1}M_{b,k}M_{d,k+1} M_{a,l}M_{c,l+1}M_{b,l}M_{d,l+1}] &= O(1/n^6) \\
\mathbb{E}[ M_{a,k} M_{a,k+1}M_{b,k}M_{c,k+1} M_{b,l}M_{c,l+1}M_{d,l}M_{d,l+1}] &= O(1/n^7) \\
\mathbb{E}[ M_{a,k} M_{c,k+1}M_{b,k}M_{d,k+1} M_{a,l}M_{b,l+1}M_{c,l}M_{d,l+1}] &= O(1/n^7) \\
\mathbb{E}[ M_{a,k} M_{a,k+1}M_{b,k}M_{c,k+1} M_{b,l}M_{c,l}M_{d,l+1}^2] &= O(1/n^6) \\
\mathbb{E}[ M_{a,k} M_{a,k+1}M_{b,k}M_{b,k+1} M_{c,l}^2M_{d,l+1}^2] &= O(1/n^5) \\
\mathbb{E}[ M_{a,k}M_{b,k} M_{c,k+1}^2 M_{a,l}M_{b,l} M_{d,l+1}^2] &= O(1/n^5) \\
\mathbb{E}[ M_{a,k}^2M_{b,k+1}^2 M_{c,l}^2 M_{d,l+1}^2] &= O(1/n^4)
\end{align*}

Going through the 8 types of terms in the expansion of $\mathbb{E}[X_{n,k}^2X_{n,l}^2]$, we see that the 2nd, 3rd, 4th, 5th, and 7th terms are all $O(1/n^3)$.  The 1st is $c^2/n^2 + O(1/n^3)$, the 6th is $-c/n^2 + O(1/n^3)$, and the 8th is $1/n^2 + O(1/n^3)$.  Note that in the expansion, we actually get two copies of terms of the 6th type since $k$ and $l$ can be switched.  

Thus, $\displaystyle{\mathbb{E}[X_{n,k}^2X_{n,l}^2] = (1-2c+c^2)/n^2 + O(1/n^3)}$.  When $|l-k| > 1$, $X_{n,k}^2X_{n,l}^2$ are all identically distributed.  This proves that $\displaystyle{\limsup_{n \to \infty} \sum_{k \neq l} \mathbb{E} [X_{n,k}^2 X_{n,l}^2] = (1-c)^2}$.   
\end{proof}

Theorem \ref{main} shows that $Tr(A M C_n M^T) \buildrel d \over \to \mathcal{N}(0, 1-c)$ where $C_n = (1,2,...,n)$.  Unlike in Section \ref{single entry}, $C_n$ is no longer a generic permutation.  The asymptotic distribution of $Tr(A M Q M^T)$ will depend on the number of fixed points of the permutation $Q$.  The following lemma generalizes Lemma \ref{Q_n} from the previous section. 
  
\begin{lemma} \label{41}
Let $A_n$ be a sequence of $n \times n$ matrices such that $Tr(A^T A) = n$ and \[\lim_{n \to \infty} \sum_{i=1}^n \frac{A_{ii}}{n} = s\] for some constant $s$.  Let $M$ be an $n \times n$ Haar distributed orthogonal matrix and \[Z \buildrel d \over = \mathcal{N}(0,1-s^2). \]  Then for every non-negative integer $f$, \[ \lim_{n \to \infty} \max_{Q_n} |\mathbb{P}\{ Tr(A M Q_n M^T) \le x \} - \mathbb{P}\{ Z+fs \le x \} | = 0 \] 
where the maximum is taken over the set of all $n \times n$ permutation matrices, $Q_n$, with $f$ fixed points and at most $2 \log n$ cycles.

\end{lemma}   

\begin{proof}
Let $Q_n$ be a permutation matrix with $f$ fixed points and $k_n \le 2 \log n$ cycles.  As before, we can assume the permutation has the form $\sigma = (1, 2,..., n_1) (n_1 + 1, n_1 + 2,...,n_2)...(n_{k_n-1} + 1, n_{k_n -1} + 2,...,n)$.  Set $X_n = Tr(A M Q_n M^T)$ and $Y_n = Tr(A M C_n M^T)$.  The difference is
\begin{equation*}
X_n - Y_n = \sum_{k=1}^n \sum_{i=1}^n \sum_{j = 1}^n A_{ij} M_{j,k}M_{i,\sigma(k)} -A_{ij} M_{j,k}M_{i,k+1}
\end{equation*} 
As in Lemma \ref{Q_n}, all the terms in the difference cancel except those involving $M_{j,n_m}$ for indices $n_m$ where the permutations $\sigma$ and $(1,...,n)$ differ.  Then we can rewrite this as
\begin{equation}
\sum_{m = 1}^{2 k_n-f} W_m + \sum_{m = 1}^{f} (V_{m,1} + V_{m,2})
\end{equation}
where \begin{equation}W_m \buildrel d \over =  \sum_{i=1}^n \sum_{j = 1}^n A_{ij} M_{j,1}M_{i,2} \end{equation} 
\begin{equation} V_{m,1} \buildrel d \over = \sum_{i \neq j} A_{ij} M_{j,1}M_{i,1}\end{equation} and \begin{equation} V_{m,2} \buildrel d \over = \sum_{i=1}^n A_{ii} M_{i,1}^2 \end{equation} 
By \eqref{variance}, we have the asymptotics  \begin{equation} \label{Wm} n \mathbb{E}[W_m^2] = O(1) \end{equation} and  \begin{equation} \label{Vm1} n \mathbb{E}[V_{m,1}^2 ] = n \sum_{i \neq j} (A_{ij}^2 + A_{ij} A_{ji}) \mathbb{E}[M_{j,k}^2M_{i,k}^2] = O(1) \end{equation}
Since $\displaystyle{\mathbb{E} \bigg[\bigg(\sum_{i=1}^n A_{ii} (M_{i,k}^2 - 1/n) \bigg)^2 \bigg] = O(1/n)}$, we have $\displaystyle{V_{m,2} \buildrel L^2 \over \rightarrow s}$ and therefore
\begin{equation} \label{Vm2} \sum_{m=1}^f V_{m,2} \buildrel d \over \rightarrow fs \end{equation}
Let $\displaystyle{Y_n' = Y_n + \sum_{m=1}^f V_{m,2}}$.  Then using \eqref{Wm}, \eqref{Vm1} and Cauchy-Schwarz, \[\mathbb{E}[(X_n- Y_n')^2] \le (2k_n)^2 (C/n)\] for sufficiently large $n$ for some constant $C$.  By Chebyshev's inequality,
\[\mathbb{P}\{ |X_n- Y_n'| > \varepsilon \} \le \varepsilon^{-2} (2k_n)^2 (C/n) \le \varepsilon^{-2} \frac{16C \log^2 n}{n} .\]   
\noindent From the proof of Lemma \ref{Q_n},
\begin{align*}
|\mathbb{P}\{ X_n \le x \} - \mathbb{P}\{ Y_n' \le x\} | &\le \mathbb{P}\{ Y_n' \le x + \varepsilon\} - \mathbb{P}\{ Y_n' \le x - \varepsilon\} + 2 \mathbb{P}\{ |X_n - Y_n'| > \varepsilon\} \\
&\le \mathbb{P}\{ Y_n' \le x + \varepsilon\} - \mathbb{P}\{ Y_n' \le x - \varepsilon\} + \varepsilon^{-2} \frac{32C \log^2 n}{n}
\end{align*}
By the triangle inequality, 
\begin{multline*}
\mathbb{P} \{ Y_n' \le x + \varepsilon\} - \mathbb{P}\{ Y_n' \le x - \varepsilon\} \le |\mathbb{P}\{ Y_n' \le x + \varepsilon\} - \mathbb{P}\{ Z+fs \le x + \varepsilon\}| \\ 
+ \mathbb{P}\{Z+fs \le x + \varepsilon\} - \mathbb{P}\{ Z+fs \le x - \varepsilon\} + |\mathbb{P}\{ Y_n' \le x - \varepsilon\} - \mathbb{P}\{ Z+fs \le x - \varepsilon\} |
\end{multline*}
Note that $\displaystyle{\lim_{n \to \infty} \sum_{i \neq j} \frac{A_{ii}A_{jj}}{n^2} = s^2}$.
Then by \eqref{Vm2} and Theorem \ref{main} (and the converging together lemma), $Y_n'$ converges weakly to $Z+fs$.  Thus, \[|\mathbb{P}\{ Y_n' \le x \} - \mathbb{P}\{ Z+fs \le x\}| < \delta/2\] for $\delta > 0$ and sufficiently large $n$.  Choose $\varepsilon$ small enough such that $\mathbb{P}\{Z+fs \le x + \varepsilon\} - \mathbb{P}\{ Z+fs \le x - \varepsilon\} < \delta/10$.  For large enough $n$, the three quantities $| \mathbb{P}\{ Y_n' \le x + \varepsilon\} - \mathbb{P}\{ Z+fs \le x + \varepsilon\}|$, $|\mathbb{P}\{ Y_n' \le x - \varepsilon\} - \mathbb{P}\{ Z+fs \le x - \varepsilon\}|$, and $\displaystyle{\varepsilon^{-2} \frac{32C \log^2 n}{n}}$ are all bounded by $\delta/10$.  Then $|\mathbb{P}\{ X_n \le x\} - \mathbb{P}\{ Y_n' \le x\}| < \delta/2$ and we have $\displaystyle{|\mathbb{P}\{ X_n \le x\} - \mathbb{P}\{ Z+fs \le x \}| < \delta}$ for sufficiently large $n$.  Since $Q_n$ was arbitrary, this proves the lemma.
\end{proof}

Putting everything together, we have the main result:
\begin{theorem} \label{main result}
Let $A_n, M, Z$ be as defined in Lemma \ref{41}.  Let $P$ be an $n \times n$ random permutation matrix and let $Y \buildrel d \over = Pois(1)$ chosen independently from $Z$.  Then in the limit of large $n$, \[Tr(AMPM^T) \buildrel d \over \rightarrow Z + sY.\] 
\end{theorem}

\begin{proof}
Let $[f]_n$ be the subset of permutations in $\mathfrak{S}_n$ with $f$ fixed points and let $F_n$ denote the number of fixed points of a random permutation in $\mathfrak{S}_n$.  Recall that $F_n \buildrel d \over \rightarrow Y$.

Let $\delta > 0$.  Choose a large integer $F$ such that $\mathbb{P}(Y > F) < \delta/5$.  By Lemma \ref{41}, for sufficiently large $n$,
\begin{align*}
& \qquad |\mathbb{P}\{ Tr(AMPM^T) \le x\} -\mathbb{P}\{ Z+sY \le x\} | \\
& = \Big|\frac{1}{n!} \sum_{f=1}^\infty \sum_{\sigma \in [f]_n} \mathbb{P}\{ Tr(A MP_{\sigma} M^T) \le x\} - \sum_{f=1}^\infty \mathbb{P}\{ Y = f\} \mathbb{P}\{ Z + fs \le x\} \Big| \\
&\begin{multlined} \le \mathbb{P}\{ F_n > F\} + \mathbb{P}\{ Y > F\} + \sum_{f=1}^F \Big|\frac{1}{n!} \sum_{\sigma \in [f]_n} \mathbb{P}\{ Tr(A MP_{\sigma} M^T) \le x\} \\ - \mathbb{P}\{ Y = f\} \mathbb{P}\{ Z + fs \le x\} \Big| \end{multlined} \\
&\begin{multlined} \le \frac{\delta}{2} + \sum_{f=1}^F \bigg( \Big|\frac{1}{n!} \sum_{\sigma \in [f]_n} \big(\mathbb{P}\{ Tr(A MP_{\sigma} M^T) \le x\} - \mathbb{P}\{ Z + fs \le x\} \big)\Big| \\ + |\mathbb{P}\{ F_n=f\} - \mathbb{P}\{ Y = f\} | \bigg) \end{multlined} \\
& \le \frac{2\delta}{3} + \frac{1}{n!} \sum_{f=1}^F \sum_{\sigma \in [f]_n} \Big| \mathbb{P}\{ Tr(A MP_{\sigma} M^T) \le x\} - \mathbb{P}\{ Z + fs \le x\} \Big| \\ 
&< \delta
\end{align*}

\end{proof}

The following corollaries illustrate a few special cases of this theorem.  

\begin{corollary} Let $k$ be a fixed positive integer.  Let $M$ be an $n \times n$ Haar-distributed orthogonal matrix and $P$ be a random $n \times n$ permutation matrix.  Then the joint distribution of $k$ entries of the random matrix $MPM^T$ normalized by $\sqrt{n}$ is asymptotically jointly i.i.d standard normal as $n \to \infty$. \end{corollary}   

\begin{corollary}
Let $A_n$ be a sequence of diagonal $n \times n$ matrices such that $Tr(A A^T) = n$ and $A_{ii} = \sqrt{1/\alpha}$ for $1 \le i \le \alpha n$ for some parameter $0 < \alpha \le 1$.  Let $Z \buildrel d \over = \mathcal{N}(0, 1 - \alpha)$ and $Y \buildrel d \over = Pois(1)$ be independent random variables.  Then \[Tr(AMPM^T) \buildrel d \over \to Z+ \sqrt{\alpha} Y.\]   
\end{corollary}

\begin{remark}
The result in Theorem \ref{main result} is for the defining representation of the symmetric group.  It would be of interest to see if similar distributional results for linear combinations of the matrix entries also hold for higher dimensional representations of the symmetric group.  
\end{remark}
 
\vspace{15pt}
\noindent \textbf{Acknowledgements:} The author wishes to thank his PhD advisor Steven Evans for introducing the problem to him and for helpful discussions and comments on the work.  He also thanks an anonymous referee for helpful suggestions and comments.

\bibliography{references}

\begin{thebibliography}{10}

\bibitem{andersonguionnet}
G.~Anderson, A.~Guionnet, and O.~Zeitouni.
\newblock {\em An introduction to random matrices}, volume 118 of {\em
  Cambridge Studies in Advanced Mathematics}.
\newblock Cambridge University Press, Cambridge, 2010.

\bibitem{borel}
E.~Borel.
\newblock Sur les principes de la theorie cin{\'e}tique des gazs.
\newblock {\em Annales de l'ecole normale sup.}, 23:9--32, 1906.

\bibitem{chatterjee}
S.~Chatterjee and E.~Meckes.
\newblock Multivariate normal approximation using exchangeable pairs.
\newblock {\em ALEA Lat. Am. J. Probab. Math. Stat.}, 4:257--283, 2008.

\bibitem{collinsmatsumoto}
B.~Collins and S.~Matsumoto.
\newblock On some properties of orthogonal {W}eingarten functions.
\newblock {\em J. Math. Phys.}, 50(113516), 2009.

\bibitem{collinssniady}
B.~Collins and P.~{\'S}niady.
\newblock Integration with respect to the {Haar} measure on unitary, orthogonal
  and symplectic group.
\newblock {\em Comm. Math. Phys.}, 264(3):773--795, 2006.

\bibitem{diaconisnewman}
A.~D'Aristotile, P.~Diaconis, and C.~Newman.
\newblock {B}rownian motion and the classical groups.
\newblock In {\em Probability, Statistics and Their Applications: Papers in
  Honor of Rabi Bhattacharya}, volume~41 of {\em IMS Lecture Notes Monogr.
  Ser.}, pages 97--116. Institute of Mathematical Statistics, 2003.

\bibitem{dey}
P.~Dey.
\newblock Limiting distribution of linear combination of coefficients of a
  representation w.r.t a random base.
\newblock Personal communication, 2008.

\bibitem{diaconisgibbs}
P.~Diaconis.
\newblock Patterns in eigenvalues: the 70th {J}osiah {W}illard {G}ibbs lecture.
\newblock {\em Bull. Amer. Math. Soc.}, 40(2):155--178, 2003.

\bibitem{diaconisfreedman}
P.~Diaconis and D.~Freedman.
\newblock A dozen de {F}inetti-style results in search of a theory.
\newblock {\em Annales de l'I. H. P. Probabilit{\'e}s et statistiques},
  23:397--423, 1987.

\bibitem{durrett}
R.~Durrett.
\newblock {\em Probability: Theory and Examples}.
\newblock Duxbury Press, 4th edition, 2005.

\bibitem{heyde}
P.~Hall and C.~Heyde.
\newblock {\em Martingale Limit Theory and Its Application}.
\newblock Academic Press, New York, NY, 1980.

\bibitem{petz}
F.~Hiai and D.~Petz.
\newblock {\em The Semicircle Law, Free Random Variables and Entropy},
  volume~77 of {\em Mathematical Surveys and Monographs}.
\newblock American Mathematical Society, Providence, RI, 2000.

\bibitem{hoeffding}
W.~Hoeffding.
\newblock A combinatorial central limit theorem.
\newblock {\em Ann. Math. Statistics.}, 22(4):558--566, 1951.

\bibitem{jiang}
T.~Jiang.
\newblock How many entries of a typical orthogonal matrix can be approximated
  by independent normals?
\newblock {\em Ann. Probab.}, 34(4):1497--1529, 2006.

\bibitem{jiang2}
T.~Jiang and Y.~Ma.
\newblock Distances between random orthogonal matrices and independent normals.
\newblock {\em Trans. Amer. Math. Soc.}, 2019.
\newblock DOI: https://doi.org/10.1090/tran/7470 (to appear).

\bibitem{mcleish}
D.~L. McLeish.
\newblock Dependent central limit theorems and invariance principles.
\newblock {\em Ann. Probab.}, 2:620--628, 1974.

\bibitem{meckes}
E.~Meckes.
\newblock Linear functions on the classical matrix groups.
\newblock {\em Trans. Amer. Math. Soc.}, 360(10):5355--5366, 2008.

\bibitem{stewart}
K.~Stewart.
\newblock Total variation approximation of random orthogonal matrices by
  {Gaussian} matrices.
\newblock arXiv:1704.06641, 2017.

\end{thebibliography}
\bibliographystyle{plain}

%\printbibliography

\end{document}